%% file: EdgeBoundDischargingFinal.tex
\documentclass[12pt]{article}
\usepackage{amsmath, amsthm, amssymb}
\usepackage{tkz-graph}
\usepackage{marginnote}
\usepackage{verbatim}
\usepackage[top=1.0in, bottom=1.0in, left=1.0in, right=1.0in]{geometry}
\usepackage{color}
\usepackage[backref=page]{hyperref}

\usepackage{sectsty}
\allsectionsfont{\sffamily}

\makeatletter
\newtheorem*{rep@theorem}{\rep@title}
\newcommand{\newreptheorem}[2]{
\newenvironment{rep#1}[1]{
 \def\rep@title{#2 \ref{##1}}
 \begin{rep@theorem}}
 {\end{rep@theorem}}}
\makeatother

\theoremstyle{plain}
\newtheorem{thm}{Theorem}[section]
\newreptheorem{thm}{Theorem}

\newreptheorem{prop}{Proposition}
\newtheorem{lem}[thm]{Lemma}
\newreptheorem{lem}{Lemma}

\newreptheorem{conjecture}{Conjecture}
\newtheorem{cor}[thm]{Corollary}
\newreptheorem{cor}{Corollary}

\theoremstyle{definition}

\theoremstyle{remark}

\newcommand{\fancy}[1]{\mathcal{#1}}

\newcommand{\IN}{\mathbb{N}}
\newcommand{\IR}{\mathbb{R}}

\newcommand{\T}{\fancy{T}}
\newcommand{\B}{\fancy{B}}
\renewcommand{\L}{\fancy{L}}
\newcommand{\HH}{\fancy{H}}

\newcommand{\set}[1]{\left\{ #1 \right\}}

\newcommand{\card}[1]{\left|#1\right|}
\newcommand{\size}[1]{\left\Vert#1\right\Vert}
\newcommand{\ceil}[1]{\left\lceil#1\right\rceil}

\newcommand{\func}[3]{#1\colon #2 \rightarrow #3}

\newcommand{\irange}[1]{\left[#1\right]}

\newcommand{\parens}[1]{\left( #1 \right)}

\newcommand{\DefinedAs}{\mathrel{\mathop:}=}

\newcommand{\AT}{\operatorname{AT}}

\newcommand{\ch}{\operatorname{ch}}

\newcommand\restr[2]{{
  \left.\kern-\nulldelimiterspace 
  #1 
  \vphantom{\big|} 
  \right|_{#2} 
  }}

\def\chil{{\chi_\ell}}
\def\chiol{\chi_{\rm{OL}}}

\newcommand{\aside}[1]{\marginnote{\scriptsize{#1}}[0cm]}
\newcommand{\aaside}[2]{\marginnote{\scriptsize{#1}}[#2]}

\title{Edge Lower Bounds for List Critical Graphs,\\ via Discharging}
\author{Daniel W. Cranston and Landon Rabern}
\author{Daniel W. Cranston\thanks{Department of Mathematics and Applied
Mathematics, Viriginia Commonwealth University, Richmond, VA;
\texttt{dcranston@vcu.edu}; 
Research of the first author is partially supported by NSA Grant
H98230-15-1-0013.}
\and
Landon Rabern\thanks{LBD Data Solutions, Lancaster, PA;
\texttt{landon.rabern@gmail.com}}
	}

\begin{document}
\maketitle
\begin{abstract}
A graph $G$ is $k$-critical if $G$ is not $(k-1)$-colorable, but every proper
subgraph of $G$ is $(k-1)$-colorable.
A graph $G$ is $k$-choosable if $G$ has an $L$-coloring from every list
assignment $L$ with $|L(v)|=k$ for all $v$, and a graph $G$ is
\emph{$k$-list-critical} if $G$ is not $(k-1)$-choosable, but every
proper subgraph of $G$ is $(k-1)$-choosable.  The problem of bounding (from below)
the number of edges in a $k$-critical graph has been widely studied, starting
with work of Gallai and culminating with the seminal results of Kostochka and
Yancey, who essentially solved the problem.  In this paper, we improve the best
lower bound on the number of edges in a $k$-list-critical graph.  Our proof
uses the discharging method, which makes it simpler and more modular than previous
work in this area.
\end{abstract}

\section{Introduction}
A $k$-coloring of a graph $G$ assigns to each vertex of $G$ a color from
$\{1,\ldots,k\}$ such that adjacent vertices get distinct colors.  A graph $G$ is
\emph{$k$-colorable} if it has a $k$-coloring and its chromatic number,
$\chi(G)$, is the least integer $t$ such that $G$ is $t$-colorable.  
Further, $G$ is \emph{$k$-critical} when $\chi(G)=k$ and every proper subgraph
$H$ of $G$ has $\chi(H)<k$.  For a graph $G$ with $\chi(G)=k$, every minimal
subgraph $H$ such that $\chi(H)=k$ must be $k$-critical.  As a result, many
questions about the chromatic number of a graph can be reduced to corresponding
questions about $k$-critical graphs.  One natural question is how few edges an
$n$-vertex $k$-critical graph $G$ can have?  Since $\delta(G)\ge k-1$, clearly
$2||G|| \ge (k-1)|G|$.  Brooks' theorem shows that if $G$ is a connected graph,
other than $K_k$, then this bound can be slightly improved.  Dirac proved that
every $k$-critical graph $G$ satisfies 
$$
2||G|| \ge (k-1)|G|+k-3.
$$
Now let 
$$
g_k(n,c) = k-1 + \frac{k-3}{(k-c)(k-1)+k-3}n.
$$
For $k\ge 4$ and $|G|\ge k+2$, Gallai improved Dirac's bound to $2||G|| \ge
g_k(|G|,0)$.  This result was subsequently strengthened by
Krivelevich~\cite{krivelevich1997minimal} to $2||G||\ge g_k(|G|,2)$ and by Kostochka and
Stiebitz~\cite{kostochkastiebitzedgesincriticalgraph}, for $k\ge 6$, to $2||G||
\ge g_k(|G|,(k-5)\alpha_k)$, where $\alpha_k = \frac12-\frac1{(k-1)(k-2)}$.
In a recent breakthrough, Kostochka and Yancey~\cite{kostochkayancey2012ore}
proved that every $k$-critical graph $G$ satisfies
$$
||G|| \ge \ceil{\frac{(k+1)(k-2)|G|-k(k-3)}{2(k-1)}}.
$$
This bound is tight for $k=4$ and $|G|\ge 6$.  Also, for each $k\ge 5$, it is tight
for infinitely many values of $|G|$.

This result of Kostochka and Yancey has numerous applications to coloring problems.
For example, it gives a short proof of Gr\"otzsch's theorem
\cite{kostochka2012oregrotsch}, that every
triangle-free planar graph is 3-colorable.  It also 
 yields short proofs of a series of results on coloring with
respect to Ore degree \cite{kierstead2009ore, rabern2010a, krs_one}. 
Thus,
it is natural to consider the same question for more general types of coloring,
such as list coloring, online list coloring, and Alon--Tarsi number (all of
which are defined below).
Gallai's bound
\cite{gallai1963kritische} also holds for list coloring, as well as online list coloring
(\cite{kostochkastiebitzedgesincriticalgraph, riasat2012critically}). 
In contrast,
Krivelevich's proof \cite{krivelevich1997minimal} does not work for list
coloring, since it uses a lemma of Stiebitz \cite{stiebitz1982proof}, which says that
in a color-critical graph, the subgraph induced by vertices of degree at least
$k$ has no more components than the subgraph induced by vertices of degree
$k-1$ (but no analogous lemma is known for list coloring). For list coloring Kostochka and Stiebitz
\cite{kostochkastiebitzedgesincriticalgraph} gave the first improvement over
Gallai's bound. 
Table \ref{tab:1}, at the end of this section, gives the values of these bounds
for small $k$. 

Recently, Kierstead and the
second author \cite{OreVizing} further improved the lower bound and extended it
to online list coloring as well as to the Alon-Tarsi number.  Their proof combined
a global averaging argument from Kostochka and Stiebitz
\cite{kostochkastiebitzedgesincriticalgraph} with improved reducibility lemmas.
 Here we use these same reducibility lemmas, but replace the global averaging
argument with a discharging argument.  The discharging argument is more
intuitive and will be easier to modify in the future for use with new
reducibility lemmas.  The improvement in our lower bound on the number of edges
in a list critical graph comes from an improved upper bound on the average
degree of Gallai trees.  To state our results we need some definitions.

List coloring was introduced by Vizing~\cite{vizing1976} and independently by
Erd\H{o}s,
Rubin, and Taylor~\cite{erdos1979choosability}.
A \emph{list assignment} $L$ assigns to each vertex $v$ of a graph $G$ a set of
allowable colors.  An \emph{$L$-coloring}\aside{$L$-coloring} is a proper
coloring $\varphi$ of $G$ such that $\varphi(v)\in L(v)$ for all $v$.  An
\emph{$f$-assignment} is a list assignment $L$ such that $|L(v)|=f(v)$ for all
$v$; a \emph{$k$-assignment} is an $f$-assignment such that $f(v)=k$ for all
$v$; a \emph{$d_0$-assignment} is an $f$-assignment such that $f(v)=d(v)$ for
all $v$.
A graph $G$ is \emph{$k$-choosable} (resp.~\emph{$f$-choosable}
\aaside{$f$-choosable}{.4cm} or
\emph{$d_0$-choosable}\aaside{$d_0$-choosable}{-.4cm}) if $G$
is $L$-colorable whenever $L$ is a $k$-assignment (resp.~$f$-assignment or
$d_0$-assignment).  The \emph{list chromatic number}, $\chil(G)$, of $G$ is the
least integer $t$ such that $G$ is $t$-choosable.  

Online list coloring allows for the possibility that the lists are being
revealed as the graph is being colored.  This notion was introduced 
independently by Zhu~\cite{zhu2009online} and Schauz~\cite{schauz2009mr} (who
called it \emph{paintability}).  A graph $G$ is \emph{online $f$-list
colorable} if either (i) $G$ is an independent
set and $f(v)\ge 1$ for all $v$ or else (ii) for every subset $S\subseteq V(G)$,
there exists an independent set $I\subseteq S$ such that $G-I$ is online
$f'$-list colorable, where $f'(v)=f(v)-1$ for all $v\in S\setminus I$ and
$f'(v)=f(v)$ for all $v\in V(G)\setminus S$.  The online list chromatic number, 
$\chiol(G)$, of $G$ is the least $k$ such that $G$ is online $f$-list colorable
when $f(v)=k$.  If $\chiol(G)\le k$, then $G$ is \emph{online $k$-list
colorable}.  Note that if $G$ is online $k$-list colorable, then $G$ is
$k$-choosable. (Given $L$, we can take $S$ to be, successively, 
$\{v:i\in L(v)\}$, as $i$ ranges through all elements of $\cup_{v\in V}L(v)$.)

A characterization of connected graphs that are not $d_0$-choosable was first
given by Vizing~\cite{vizing1976}, and later by Erd\H{o}s, Rubin, and
Taylor~\cite{erdos1979choosability}.  Such graphs are called Gallai trees;
they are precisely the connected graphs in which each block is a complete graph
or an odd cycle.  Hladk{\`y}, Kr{\'a}l, and Schauz~\cite{Hladky} later
characterized the connected
graphs that are not online $d_0$-choosable; again, these are precisely the
Gallai trees.

Given a graph $G$ and a list assignment $L$, a natural way to construct an
$L$-coloring of $G$ is to color $G$ greedily in some order.  This approach will
always succeed when $|L(v)|\ge d(v)+1$ for all $v$.  Using digraphs, we can
state a weaker sufficient condition.  For a vertex ordering $\sigma$, form an
acyclic digraph $D$ by directing each edge $v_iv_j$ as $v_i \to v_j$ if $v_j$
precedes $v_i$ in $\sigma$.  Now it suffices to have $|L(v)|\ge d^+_D(v)+1$. 
Alon and Tarsi strengthened this result significantly, by
allowing certain directed cycles in $D$.  To state their result, we need a few
definitions.  A digraph $D$ is \emph{eulerian} if $d^+_D(v)=d^-_D(v)$ for all
$v$.  A digraph $D$ is \emph{even} if $||D||$ is even, and otherwise $D$ is
\emph{odd}.  For a digraph $D$, let $EE(D)$ and $EO(D)$ denote the number of
even (resp.~odd) spanning eulerian subdigraphs of $D$.  A graph $G$ is
\emph{$f$-Alon--Tarsi} (\emph{$f$-AT}\aside{$f$-AT}, for short) if $G$ has an
orientation $D$ such that $EE(D)\ne EO(D)$ and $f(v)\ge d^+_D(v)+1$ for all
$v$.  The Alon--Tarsi number, $AT(G)$, of $G$ is the least $k$ such that $G$ is
$f$-AT when $f(v)=k$ for all $v$.  Analogous to the definition for coloring and
list-coloring, a graph $G$ is \emph{$k$-AT-critical}\aside{$k$-AT-critical} if
$\AT(G)=k$ and $\AT(H)<k$ for every proper subgraph $H$ of $G$.  From the
definitions, it is easy to check that always $\chi(G)\le \chil(G)\le
\chiol(G)\le AT(G)\le \Delta(G)+1$.
Alon and Tarsi~\cite{Alon1992125} proved the following.
 
\begin{lem}\label{AlonTarsi}
If a graph $G$ is $f$-AT for $\func{f}{V(G)}{\IN}$, then $G$ is $f$-choosable.
\end{lem}

Note that always $EE(D)\ge 1$, since the edgeless spanning subgraph is even. 
If $D$ is acyclic, then $EO(D)=0$, since every subgraph with edges has a vertex
$v$ with $d^+(v)\ge 1 > 0 = d^-(v)$.  Thus, Lemma~\ref{AlonTarsi} generalizes the
results we can prove by greedy coloring.  Schauz \cite{schauz2010flexible} gave
a new, constructive proof of this lemma (the original was non-constructive),
which allowed him to extend the result to online $f$-choosability.

\begin{lem}\label{Schauz}
If a graph $G$ is $f$-AT for $\func{f}{V(G)}{\IN}$, then $G$ is online $f$-choosable.
\end{lem}

\begin{table}
	\begin{center}
		\begin{tabular}{|c|c|c|c|c|c|c|c|}
			\hline
			&\multicolumn{4}{ |c| }{$k$-Critical
$G$}&\multicolumn{3}{|c|}{$k$-List Critical $G$}\\
			\hline
			& Gallai \cite{gallai1963kritische}
			& Kriv \cite{krivelevich1997minimal}
			& KS \cite{kostochkastiebitzedgesincriticalgraph}
			& KY \cite{kostochkayancey2012ore}
			& KS \cite{kostochkastiebitzedgesincriticalgraph} 
			& KR \cite{OreVizing}
			& Here\\
			$k$ & $d(G) \ge$ & $d(G) \ge$ & $d(G) \ge$ & $d(G) \ge$ & $d(G) \ge$ & $d(G) \ge$ & $d(G) \ge$\\
			\hline 
			4 & 3.0769 &3.1429&---&3.3333& --- & --- & \bf{---}\\
			5 & 4.0909 &4.1429&---&4.5000& --- & 4.0984 & \bf{4.1000}\\
			6 & 5.0909 &5.1304&5.0976&5.6000& --- & 5.1053 & \bf{5.1076}\\
			7 & 6.0870 &6.1176&6.0990&6.6667& --- & 6.1149 & \bf{6.1192}\\
			8 & 7.0820 &7.1064&7.0980&7.7143& --- & 7.1128 & \bf{7.1167}\\
			9 & 8.0769 &8.0968&8.0959&8.7500& 8.0838 & 8.1094 & \bf{8.1130}\\
			10 & 9.0722 &9.0886&9.0932&9.7778& 9.0793 & 9.1055 & \bf{9.1088}\\
			15 & 14.0541 &14.0618&14.0785&14.8571& 14.0610 & 14.0864 & \bf{14.0884}\\
			20 & 19.0428 &19.0474&19.0666&19.8947& 19.0490 & 19.0719 & \bf{19.0733}\\
			\hline
		\end{tabular}
	\end{center}
	\caption{History of lower bounds on the average degree $d(G)$ of
$k$-critical and $k$-list-critical graphs $G$. (Reproduced from~\cite{OreVizing}
and updated.)}
	\label{tab:1}
\end{table}

In this paper, we prove lower bounds on the number of edges in a $k$-AT-critical
graph.  Corollaries~\ref{MainCor} and~\ref{MinorCor} summarize our main results.
Here $d(G)$ denotes the average degree of $G$.

\begin{repcor}{MainCor}
If $G$ is a $k$-AT-critical graph, with $k\ge 7$, and $G\ne K_k$, then
\[d(G) \ge k-1 + \frac{(k-3)(2k-5)}{k^3 + k^2 - 15k + 15}.\]
\end{repcor}

\begin{repcor}{MinorCor}
If $G$ is a $k$-AT-critical graph, with $k\in\{5,6\}$, and $G\ne K_k$, then
\[d(G) \ge k-1 + \frac{(k-3)(2k-5)}{k^3 + 2k^2 - 18k + 15}.\]
\end{repcor}

\section{Gallai's bound via discharging}
\label{sec:gallai}

One of the earliest results bounding the number of edges in a critical graph is
the following theorem, due to Gallai.  The key lemma he proved,
Lemma~\ref{BasicGallaiTreeBound}, gives an upper bound on the number of edges
in a Gallai tree.  The rest of his proof is an easy
counting argument.  As a warmup, and to illustrate the approach that we take in
Section~\ref{discharging}, we rephrase this counting in terms of
discharging.  A \emph{Gallai tree} is a connected graph in which each block is a
clique or an odd cycle (these are precisely the $d_0$-choosable graphs,
mentioned in the introduction).  Let $\T_k$ denote the set of all Gallai
trees of degree at most $k-1$, excluding $K_k$. A
\emph{$k$-vertex}\aside{$k$-vertex} (resp.~$k^+$-vertex or $k^-$-vertex) is a
vertex of degree $k$ (resp.~at least $k$ or at most $k$).  A
\emph{$k$-neighbor}\aside{$k$-neighbor} of a vertex $v$ is an adjacent $k$-vertex.
For convenience, we write $d(G)$\aaside{$d(G)$}{.5cm} to denote the average degree of $G$.

\begin{thm}[Gallai]
\label{thm:Gallai}
	For $k \ge 4$ and $G$ a $k$-AT-critical graph, with $G \ne K_k$, we have
	\[d(G) > k-1 + \frac{k-3}{k^2-3}.\]
\end{thm}
\begin{proof}
We use the discharging method. Each vertex $v$ has initial charge $d_G(v)$. 
First, each $k^+$-vertex gives charge $\frac{k-1}{k^2-3}$ to each of its
$(k-1)$-neighbors.  Now the vertices in each component of the subgraph induced
by $(k-1)$-vertices share their total charge equally.  Let $\ch^*(v)$ denote the
resulting charge on $v$.  We finish the proof by showing that $\ch^*(v) \ge k-1
+ \frac{k-3}{k^2-3}$ for all $v \in V(G)$.
	
If $v$ is a $k^+$-vertex, then $ch^*(v) \ge d_G(v) - \frac{k-1}{k^2-3}d_G(v) =
\parens{1- \frac{k-1}{k^2-3}}d_G(v) \ge \parens{1- \frac{k-1}{k^2-3}}k = k-1 +
\frac{k-3}{k^2-3}$ as desired.

Instead, let $T$ be a component of the subgraph induced by $(k-1)$-vertices. 
Now the vertices in $T$ receive total charge 
\[\frac{k-1}{k^2-3}\sum_{v \in V(T)} (k-1 - d_T(v)) =
\frac{k-1}{k^2-3}\parens{(k-1)|T| - 2\size{T}}.\]
So, after distributing this charge equally, each vertex in $T$ receives charge
\[\frac{1}{|T|}\frac{k-1}{k^2-3}((k-1)|T| - 2\size{T}) = \frac{k-1}{k^2-3}\parens{(k-1) - d(T)}.\]
By Lemma \ref{BasicGallaiTreeBound}, which we prove next, this is greater than
\[\frac{k-1}{k^2-3}\parens{(k-1) - \parens{k-2 + \frac{2}{k-1}}} = \frac{k-1}{k^2-3}\parens{\frac{k-3}{k-1}} = \frac{k-3}{k^2-3}.\]
Hence, each $(k-1)$-vertex ends with charge greater than $k-1 +
\frac{k-3}{k^2-3}$, as desired.
\end{proof}

\begin{lem}[Gallai]
\label{BasicGallaiTreeBound}
	For $k \ge 4$ and $T \in \T_k$, we have $d(T) < k-2 + \frac{2}{k-1}$.
\end{lem}
\begin{proof}
	Suppose the lemma is false and choose a counterexample $T$ minimizing
$|T|$.  Now $T$ has at least two blocks.  Let $B$ be an endblock of $T$.  If $B$
is $K_t$ for some $t\in \{2,\ldots, k-2\}$, then remove the non-cut vertices of
$B$ from $T$ to get $T'$.  By the minimality of $|T|$, we have 
	\[2\size{T} - t(t-1) = 2\size{T'} < \parens{k-2 + \frac{2}{k-1}}|T'| = \parens{k-2 + \frac{2}{k-1}}\parens{|T|-(t-1)}.\]
	Hence, we have the contradiction
	\[2\size{T} < \parens{k-2 + \frac{2}{k-1}}|T| + (t+2 -k - \frac{2}{k-1})(t-1) \le \parens{k-2 + \frac{2}{k-1}}|T|.\]
	
	The case when $B$ is an odd cycle is similar to that above, when $t=3$; a
longer cycle just makes the inequality stronger.  Finally, if $B = K_{k-1}$,
remove all vertices of $B$ from $T$ to get $T'$. By the minimality of $|T|$, we have 
	\begin{align*}
	  2\size{T} - (k-1)(k-2) - 2 &= 2\size{T'}\\
	  &< \parens{k-2 + \frac{2}{k-1}}|T'|\\
	  &= \parens{k-2 + \frac{2}{k-1}}|T| - \parens{k-2 + \frac{2}{k-1}}(k-1).
	\end{align*}

	Hence, $2\size{T} < \parens{k-2 + \frac{2}{k-1}}|T|$, a contradiction.
\end{proof}

\section{A refined bound on \texorpdfstring{$||T||$}{||T||}}
Lemma~\ref{BasicGallaiTreeBound} is essentially best possible, as shown by a
path of copies of $K_{k-1}$, with each successive pair of copies linked by a
copy of $K_2$.  When the path $T$ has $m$ copies of $K_{k-1}$, we get
$2||T||=m(k-1)(k-2)+2(m-1) = (k-2+\frac2{k-2})|T|-2$.  And a small modification
to the proof above yields $2||T|| \le (k-2+\frac2{k-2})|T|-2$. 
Fortunately, this is not the end of the story.

We see two potential places that we could improve the bound in
Theorem~\ref{thm:Gallai}. For each graph $G$, we could show that either (i) the
bound in Lemma~\ref{BasicGallaiTreeBound} is loose or (ii) many of the
$k^+$-vertices finish with extra charge, because they have incident edges
leading to other $k^+$-vertices (rather than only $(k-1)$-vertices, as allowed
in the proof of Theorem~\ref{thm:Gallai}).  A good way to quantify this
slackness in the proof is with the parameter $q(T)$\aside{$q(T)$}, which
denotes the number of non-cut vertices in $T$ that appear in copies of
$K_{k-1}$.  When $q(T)$ is small relative to $|T|$, we can save as in (i)
above.  And when it is large, we can save as in (ii).  In the direction of (i),
we now prove a bound on $||T||$ akin to that in
Lemma~\ref{BasicGallaiTreeBound}, but which is stronger when
$q(T)\le|T|\frac{k-3}{k-1}$.  In Section~\ref{discharging} 
we do the discharging; at that point we handle case (2),
using a reducibility lemma proved in~\cite{OreVizing}. 

Without more reducible configurations we can't hope to prove $d(T) < k-3$, since
each component $T$ could be a copy of $K_{k-2}$.  This is why our next bound on
$2||T||$ has the form $(k-3 + p(k))|T|$; since we will always have $p(k)>0$,
this is slightly worse than average degree $k-3$.  To get the best edge bound
we will take $p(k)=\frac{3k-5}{k^2 - 4k + 5}$, but we prefer to prove the more general
formulation, which shows that previous work of Gallai \cite{gallai1963kritische}
and Kostochka and Steibitz~\cite{kostochkastiebitzedgesincriticalgraph} fits
the same pattern.  This general version will also be more
convenient for the discharging.  %
It is helpful to handle separately the cases $K_{k-1}\not\subseteq T$ and
$K_{k-1}\subseteq T$.  The former is simpler, since it implies $q(T)=0$, so we
start there.

\begin{lem}\label{BoundFamilyWithoutKKMinusOne}
	Let $\func{p}{\IN}{\IR}$, $\func{f}{\IN}{\IR}$.
	For all $k \ge 5$ and $T \in \T_k$ with $K_{k-1} \not \subseteq T$, we have
	\[2\size{T} \le (k-3 + p(k))\card{T} + f(k)\]
	whenever $p$ and $f$ satisfy all of the following conditions:
	\begin{enumerate}
		\item $p(k) \ge \frac{-f(k)}{k-2}$; and
		\item $p(k) \ge \frac{-f(k)}{5} + 5 - k$; and
		\item $0\ge f(k)\ge -k+2$; and
		\item $p(k) \ge \frac{3}{k-2}$.
	\end{enumerate}
\end{lem}
\begin{proof}
A general outline for the proof is that it mirrors that of
Lemma~\ref{BasicGallaiTreeBound}, and we add as hypotheses all of the conditions that
we need along the way.

Suppose the lemma is false and choose a counterexample $T$ minimizing $|T|$. 
If $T$ is $K_t$ for some $t \in \{2,k-2\}$, then $t(t-1) > (k-3 + p(k))t +
f(k)$.  After substituting $p(k)\ge \frac{-f(k)}{k-2}$ from (1), this
simplifies to $-t(k-2)>f(k)$, which contradicts (3).  If $T$ is $C_{2r+1}$ for
$r \ge 2$, then $2(2r+1) > (k-3 + p(k))(2r+1) + f(k)$ and hence
$(5-k-p(k))(2r+1)>f(k)$.  Since $f(k)\le 0$, this contradicts (2).  (Note that
we only use conditions (1), (2), and (3) when $T$ has a single block;
these are the base cases when the proof is phrased using induction.)

Let $D$ be an 
induced subgraph such that $T\setminus D$ is connected.  (We will choose $D$ to
be a connected subgraph contained in at most three blocks of $T$.)
Let $T' = T \setminus D$. 
By the minimality of $|T|$, we have
	\[2\size{T'} \le (k-3 + p(k))\card{T'} + f(k).\]
	Since $T$ is a counterexample, subtracting this inequality from the inequality for
$2||T||$ gives
	\begin{equation}
	2\size{T} - 2\size{T'} > (k-3 + p(k))|D|. \tag{*}\label{eqn:*}
	\end{equation}
	
Suppose $T$ has an endblock $B$ that is $K_t$ for some $t \in \{3,\ldots,
k-3\}$; let $x_B$ be a cut vertex of $B$ and let $D=B-x_B$.
Now \eqref{eqn:*} gives $2\size{T}-2\size{T'} =
\card{B}(\card{B}-1)>(k-3+p(k))(|B|-1)$, which 
is a contradiction, since $|B|\le k-3$ and $p(k)>0$.
Suppose instead that $T$ has an endblock $B$ that is an odd cycle.  Again, let
$D=B-x_B$.  Now we get $2|B|>(k-3+p(k))(|B|-1)$.  This simplifies to $|B|<1+\frac2{k-5+p(k)}$, which is a contradiction, 
since the denominator is always at least 1 (using (4) when $k=5$).
Finally suppose that $T$ has an endblock $B$ that is $K_2$. Now \eqref{eqn:*} gives
$2 > k-3 + p(k)$, which is again a contradiction, since $k \ge 5$ and $p(k) > 0$.
	
To handle the case when $B$ is $K_{k-2}$ we need to remove $x_B$ from $T$ as
well, so we simply let $D=B$.  
Since $B=K_{k-2}$, we have either $d_T(x_B) = k - 2$ or $d_T(x_B) =
k-1$. When $d_T(x_B) = k - 2$, we have
	\[(k-2)(k-3) +2 > (k-3 + p(k))(k-2),\]
	contradicting (4).
	
The only remaining case is when $B$ is $K_{k-2}$ and $d_T(x_B) =
k - 1$.  Each case above applied when $B$ was any endblock of $T$, so we may
assume that every endblock of $T$ is a copy of $K_{k-2}$ that shares a vertex
with an odd cycle.  Choose an endblock $B$ that is the end of a longest path in
the block-tree of $T$.  Let $C$ be the odd cycle sharing a vertex $x_B$ with
$B$.  Consider a neighbor $y$ of $x_B$ on $C$ that either (i) lies only in $C$
or (ii) lies also in an endblock $A$ that is a copy of $K_{k-2}$ (such a
neighbor exists because $B$ is at the end of a longest path in the block-tree).
In (i), let $D=B\cup\{y\}+yx_B$; in (ii), let $D=B\cup A+yx_B$.

In (i), equation \eqref{eqn:*} gives

\[(k-2)(k-3)+2(3) > (k-3+p(k))(k-1).\]
This simplifies to $6>k-3+(k-1)p(k)$, and eventually, by (4), to $6>k+\frac3{k-2}$, which yields a contradiction.

In (ii), equation \eqref{eqn:*} gives
	\[2(k-2)(k-3) + 2(3) > 2(k-3 + p(k))(k-2),\]
	which simplifies to
	\[3 > (k-2)p(k),\]
	again contradicting (4).
\end{proof}

Lemma \ref{BoundFamilyWithoutKKMinusOne} gives the tightest bound on $||T||$
when $p(k) = \frac{3}{k-2}$ and $f(k) = -3$.   However, for the discharging in
Section~\ref{discharging}, it will be convenient to apply Lemma
\ref{BoundFamilyWithoutKKMinusOne} with a larger $p(k)$, to match the best value
of $p(k)$ that works in the analogous lemma for $K_{k-1}\subseteq T$.  We now
prove such a lemma.  Its statement is similar to the
previous one, but with an extra term in the bound, as well as slightly different hypotheses.

\begin{lem}\label{BoundFamilyWithKKMinusOne}
	Let $\func{p}{\IN}{\IR}$, $\func{f}{\IN}{\IR}$, $\func{h}{\IN}{\IR}$. 
	For all $k \ge 5$ and $T \in \T_k$ with $K_{k-1} \subseteq T$, we have
	\[2\size{T} \le (k-3 + p(k))\card{T} + f(k) + h(k)q(T)\]
	whenever $p$, $f$, and $h$ satisfy all of the following conditions:
	\begin{enumerate}
		\item $f(k) \ge (k-1)(1- p(k) - h(k))$; and	
	    \item $p(k) \ge \frac{3}{k-2}$; and
		\item $p(k) \ge h(k) + 5 - k$; and
		\item $p(k) \ge \frac{2+h(k)}{k-2}$; and
		\item $(k-1)p(k) + (k-3)h(k) \ge k+1$.
	\end{enumerate}
	
\end{lem}
\begin{proof}

The proof is similar to that of Lemma~\ref{BoundFamilyWithoutKKMinusOne}.  The
main difference is that now our only base case is $T=K_{k-1}$.  For this
reason, we replace hypotheses (1), (2), and (3) of
Lemma~\ref{BoundFamilyWithoutKKMinusOne}, which we used only for the base cases
of that proof, with our new hypothesis (1), which we use for the current base
case.  When some endblock $B$ is an odd cycle or $K_t$, with $t\in\{3, \ldots,
k-3\}$, the induction step is identical to that in
Lemma~\ref{BoundFamilyWithoutKKMinusOne}, since deleting $D$ does not change
$q(T)$.
It is easy to check that, as needed, $K_{k-1}\subseteq T\setminus D$.  Thus, we
need only to consider the induction step when $T$ has an endblock $B$ that is
$K_2$, $K_{k-2}$, or $K_{k-1}$.  As we will see, these three cases require
hypotheses (3), (4), and (5), respectively.
    
Let $T$ be a counterexample minimizing $|T|$.  Let $D$ be an induced subgraph
such that $T\setminus D$ is connected, and let $T'=T\setminus D$. The same
argument as in Lemma~\ref{BoundFamilyWithoutKKMinusOne} now gives
    \begin{equation}
		2||T||-2||T'|| > (k-3 + p(k))\card{D} + h(k)\parens{q(T) -
q(T')}.\tag{**}\label{eqn:**}
	\end{equation}
If $B$ is $K_2$, then $q(T') \le q(T) + 1$ and \eqref{eqn:**} gives $2 > k-3 +
p(k) - h(k)$, contradicting (3).
So every endblock of $B$ is $K_{k-2}$ or $K_{k-1}$. To handle these cases, we
will need to remove $x_B$ from $T$ as well.  Suppose some endblock $B$ is
$K_{k-1}$ and $K_{k-1} \subseteq T\setminus B$.  Let $D=B$.  Now
$q(T') \le q(T)-(k-2)+1$.  So \eqref{eqn:**} gives
	
\[ (k-1)(k-2)+2 > (k-3+p(k))(k-1)+h(k)(k-3).\]
This simplifies to $k+1 > (k-1)p(k)+(k-3)h(k)$, which contradicts (5).  Thus,
at most one endblock of $T$ is $K_{k-1}$.
Since the cases above apply when $B$ is any endblock, each other endblock must
be $K_{k-2}$.  Let $B$ be such an endblock, and $x_B$ its cut vertex.	So
$d_T(x_{B}) = k - 2$ or $d_T(x_{B}) = k-1$.  In the former case, $q(T') \le
q(T) + 1$, and in the latter, $q(T) = q(T')$.
If $d_T(x_{B}) = k - 2$, then \eqref{eqn:**} gives
\[(k-2)(k-3) +2 > (k-3 + p(k))(k-2) - h(k),\]
which simplifies to $\frac{2+h(k)}{k-2} > p(k)$, and contradicts (4).
	
Hence, all but at most one endblock of $T$ is a copy of $K_{k-2}$ with a
cut vertex that is also in an odd cycle.  Let $B$ be such an endblock 
at the end of a longest path in the block-tree of $T$, and let $C$ be the odd
cycle sharing a vertex $x_B$ with $B$.  Consider a neighbor $y$ 
of $x_B$ on $C$ that either (i) lies only in block $C$ or (ii) lies also in an
endblock $A$ that is a copy of $K_{k-2}$ (such a neighbor exists 
because $B$ is at the end of a longest path in the block-tree).  In (i), let
$D=B\cup\{y\}+yx_B$; in (ii), let $D=B\cup A+yx_B$.  Let $T'=T\setminus
V(D)$.	In each case, we have $q(T')=q(T)$, so the analysis is identical to that
in the proof of Lemma~\ref{BoundFamilyWithoutKKMinusOne}.
\end{proof}

Now let's see some examples of using Lemma \ref{BoundFamilyWithoutKKMinusOne}
and Lemma \ref{BoundFamilyWithKKMinusOne}.  What happens if we take $h(k) = 0$
in Lemma \ref{BoundFamilyWithKKMinusOne}?  By hypothesis (5), we need
$(k-1)p(k) \ge k + 1$ and hence $p(k) \ge 1 + \frac{2}{k-1}$.  Taking $p(k) = 1 +
\frac{2}{k-1}$, (1) requires $f(k) \ge -2$.  Using $h(k)=0$,
$p(k)=1+\frac2{k-1}$, and $f(k) = -2$, all of the conditions are satisfied in
both of Lemmas~\ref{BoundFamilyWithoutKKMinusOne} and
\ref{BoundFamilyWithKKMinusOne}, so we conclude $2\size{T} \le \parens{k-2 +
\frac{2}{k-1}}\card{T} - 2$ for every $T \in \T_k$ when $k \ge 5$.  This is the
previously mentioned slight refinement of Gallai's Lemma \ref{BasicGallaiTreeBound}.

Instead, let's make $p(k)$ as small as Lemma \ref{BoundFamilyWithKKMinusOne}
allows. By (4), $h(k) \le (k-2)p(k) - 2$. Plugging this into (5) and solving,
we get $p(k) \ge \frac{3k-5}{k^2 - 4k + 5}$.  Now $\frac{3k-5}{k^2 - 4k + 5}
\ge \frac{3}{k-2}$ for $k \ge 5$, so $p(k) = \frac{3k-5}{k^2 - 4k + 5}$
satisfies (2).  With $h(k) = \frac{k(k-3)}{k^2 - 4k + 5}$, we also satisfy (3),
(4), and (5).  Now with $f(k) = -\frac{2(k-1)(2k-5)}{k^2 - 4k + 5}$, condition
(1) is satisfied, so by Lemma~\ref{BoundFamilyWithKKMinusOne} we have the following.

\begin{cor}\label{SmallP}
	For $k \ge 5$ and $T \in \T_k$ with $K_{k-1} \subseteq T$, we have
	\[2\size{T} \le \parens{k-3 + \frac{3k-5}{k^2 - 4k + 5}}\card{T} - \frac{2(k-1)(2k-5)}{k^2 - 4k + 5} + 
	\frac{k(k-3)}{k^2 - 4k + 5}q(T).\]
\end{cor}

If we put a similar bound of Kostochka and
Stiebitz~\cite{kostochkastiebitzedgesincriticalgraph} into this form, we get
the following.
\begin{lem}[Kostochka--Stiebitz]

\input{m2}
\input{m3}

		For $k \ge 7$ and $T \in \T_k$, we have
		\[2\size{T} \le \parens{k-3 + \frac{4(k-1)}{k^2 - 3k + 4}}\card{T} - \frac{4(k^2-3k+2)}{k^2-3k+4} + 
		\frac{k^2 - 3k}{k^2-3k+4}q(T).\]
\end{lem}
\noindent
Note that $\frac{3k-5}{(k-5)(k-1)} < \frac{4(k-1)}{k^2 - 3k + 4}$ for $k \ge 7$.

In Section~\ref{discharging}, we will see that the bound we get on $d(G)$ is
primarily a function of the $p(k)$ with which we apply
Lemma~\ref{BoundFamilyWithKKMinusOne}: the smaller $p(k)$ is, the better bound
we get on $d(G)$.  So it useful to note that the choice of $p(k)$ in
Corollary~\ref{SmallP} is best possible.  We now give a construction to prove
this.  Let $X$ be a $K_{k-1}$ with $k-3$ pendant edges.  At the end of each
pendant edge, put a $K_{k-2}$.  Make a path of copies of $X$ by adding one edge
between the $K_{k-1}$ in each copy of $X$ (in the only way possible to keep the
degrees at most $k-1$).  Let $T$ be the path made like this from $m$ copies of
$X$.  Now $q(T) = 2$ (from the copies of $K_{k-1}$ at the ends of the path), so
if $T$ satisfies the bound in Lemma~\ref{BoundFamilyWithKKMinusOne}, then we
must have
\begin{align*}
&m((k-1)(k-2) + (k-3)(k-2)(k-3) + 2(k-3)) + 2(m-1) \\
\le &(k-3 + p(k))m(k-1+(k-2)(k-3)) + f(k) + 2h(k),
\end{align*}
which reduces to
\[m(k-1) + 2m(k-3) + 2(m-1) \le m(k-1+(k-2)(k-3))p(k) + f(k) + 2h(k).\]
Now solving for $p(k)$ gives
\[p(k) \ge \frac{m(k-1)+2m(k-3)+2(m-1)-f(k)-2h(k)}{m((k-1)+(k-2)(k-3))},\] 
which simplifies to
\[p(k) \ge \frac{3k - 5}{k^2 - 4k + 5}-\frac{2+f(k)+2h(k)}{m(k^2-4k+5)}.\]
Since we can make $m$ arbitrarily large, this implies $p(k)\ge \frac{3k-5}{k^2-4k+5}$, as desired.

\section{Discharging}\label{discharging}

\subsection{Overview and Discharging Rules}
\label{discharging-overview}

Now we use the discharging method, together with the edge bound lemmas of the
previous section, to give an improved bound on $d(G)$ for every $k$-critical
graph $G$.  It is helpful to view our proof here as a refinement and
strengthening of the proof of Gallai's bound, in Section~\ref{sec:gallai}.  For
$T \in \T_k$, let $W^k(T)$\aside{$W^k(T)$} be the set of vertices of $T$ that
are contained in some $K_{k-1}$ in $T$.  For a $k$-AT-critical graph $G$, let
$\L(G)$\aside{$\L(G)$} denote the subgraph of $G$ induced on the
$(k-1)$-vertices and $\HH(G)$\aside{$\HH(G)$} the subgraph of $G$ induced on
the $k$-vertices.     

Note that in the proof of Gallai's bound, all $(k+1)^+$-vertices finish with
extra charge; $(k+1)$-vertices have extra charge almost 1 and vertices of
higher degree have even more.  Our idea to improve the bound on $d(G)$ is to
have the $k$-vertices give slightly less charge, $\epsilon$, to their
$(k-1)$-neighbors.  Now all $k^+$-vertices finish with extra charge.  But
components of $\L(G)$ have less charge, so we need to give them more charge
from $(k+1)^+$-neighbors.  How much charge will each component $T$ of $\L(G)$
receive? This depends on $||T||$.  If $||T||$ is small, then $T$ has many
external neighbors, so $T$ will receive lots of charge.  If $||T||$ is large,
then Lemma 3.2 implies that $q(T)$ is also large.  So our plan is to send
charge $\gamma$ to $T$ via each edge incident to a vertex in $W^k(T)$, i.e.,
one counted by $q(T)$.  (For comparison with Gallai's bound, we will have
$\epsilon < \frac{k-1}{k^2-3} < \gamma$.)  If such an incident edge ends at a
$(k+1)^+$-vertex $v$, then $v$ will still finish with sufficient charge.  Our
concern, of course, is that a $k$-vertex will give charge $\gamma$ to too many
vertices in $W^k(T)$.  We would like to prove that each $k$-vertex has only a
few neighbors in $W^k(T)$.  Unfortunately, we believe this is false. However,
something similar is true.  We can assign each $k$-vertex to ``sponsor'' some
adjacent vertices in $W^k(T)$, so that each
$k$-vertex sponsors at most 3 such neighbors, and in each component $T$ of
$\L(G)$ at most two vertices in $W^k(T)$ go unsponsored.  This is an
immediate consequence of Lemma \ref{MultipleHighConfigurationEuler}, which says
that the auxiliary bipartite graph $\B_k(G)$, defined in 
the next paragraph, is 2-degenerate.  Now we give the details.

Let $\B_k(G)$\aside{$\B_k(G)$} be the bipartite graph with one part $V(\HH(G))$
and the other part the components of $\L(G)$.  Put an edge between $y \in
V(\HH(G))$ and a component $T$ of $\L(G)$ if and only if $N(y) \cap W^k(T) \ne
\emptyset$.  Now Lemma~\ref{MultipleHighConfigurationEuler} says that $\B_k(G)$
is $2$-degenerate.
Let $\epsilon$ and $\gamma$\aside{$\epsilon, \gamma$} be parameters, to be
chosen. Our initial charge function is $\ch(v) = d_G(v)$.  We
redistribute charge according to the following rules, applied successively.  
\begin{enumerate}
	\item Each $k^+$-vertex gives charge $\epsilon$ to each of its $(k-1)$-neighbors not in a $K_{k-1}$.
	\item Each $(k+1)^+$-vertex give charge $\gamma$ to each of its $(k-1)$-neighbors in a $K_{k-1}$.
	\item Let $Q = \B_k(G)$.  Repeat the following steps until $Q$ is empty.
	  \begin{enumerate}
	  	\item For each component $T$ of $\L(G)$ in $Q$ with degree at most two in $Q$ do the following:
	  	    \begin{enumerate}
	  	    	\item For each $v \in V(\HH(G)) \cap V(Q)$ such that $\card{N_G(v) \cap W^k(T)} = 2$, pick one $x \in N_G(v) \cap W^k(T)$ and send charge $\gamma$ from $v$ to $x$,
	  	    	\item Remove $T$ from $Q$.
	  	    \end{enumerate}
	  	\item For each vertex $v$ of $\HH(G)$ in $Q$ with degree at most two in $Q$ do the following:  
	  	    \begin{enumerate}
	  	        \item Send charge $\gamma$ from $v$ to each $x \in N_G(v) \cap W^k(T)$ for each component $T$ of $\L(G)$ where $vT \in E(Q)$.
	  	        \item Remove $v$ from $Q$.
	        \end{enumerate}
	  \end{enumerate}
	\item Have the vertices in each component of $\L(G)$ share their total charge equally.
\end{enumerate}

First, note that Step 3 will eventually result in Q being empty.  This is
because $\B_k(G)$ is 2-degenerate, as shown in
Lemma~\ref{MultipleHighConfigurationEuler}.  Next, consider a $k$-vertex $v$. 
In (3bi) $v$ gives away $\gamma$ to each neighbor in at most two components of
$\L(G)$.  So it is important that $v$ have few neighbors in these components. 
Fortunately, this is true.  By Lemma~\ref{ConfigurationTypeOneEuler}, $v$ has
at most 2 neighbors in any component of $\L(G)$.  Further, $v$ has at most one
component in which  it has 2 neighbors.  Thus, in (3ai) and (3bi), $v$ gives
away a total of at most $3\gamma$.  Finally, consider a component $T$.  In
(3bi), $T$ receives charge $\gamma$ via every edge incident in $\B_k(G)$,
except possibly two (that are still present when $v$ is deleted in (3aii)). 
Again, by Lemma~\ref{MultipleHighConfigurationEuler}, no such $v$ has three
neighbors in $T$.  Further, combining this with Steps (2) and (3ai), $T$
receives $\gamma$ along all but at most two incident edges leading to
$k$-vertices.  Thus, $T$ receives charge at least $\gamma(q(T)-2)$ in Steps (2)
and (3).

\subsection{Analyzing the Discharging and the Main Result}
\label{discharging-analyzing}
Here we analyze the charge received by each component $T$ of $\L(G)$.  We
choose $\epsilon$ and $\gamma$ to maximize the minimum, over all vertices, of
the final charge.  The following theorem is the main result of this paper.

\begin{thm}\label{UberTheorem}
	Let $k \ge 7$ and $\func{p}{\IN}{\IR}$, $\func{f}{\IN}{\IR}$, $\func{h}{\IN}{\IR}$.  If $G$ is a $k$-AT-critical graph, and $G\ne K_k$, then 
	\[d(G) \ge k-1 + \frac{2-p(k)}{k+2 + 3h(k) - p(k)},\]
	whenever $p$, $f$, and $h$ satisfy all of the following conditions:
	\begin{enumerate}
	\item $f(k) \ge (k-1)(1- p(k) - h(k))$; and	
	    \item $p(k) \ge \frac{3}{k-2}$; and
		\item $p(k) \ge h(k) + 5 - k$; and
		\item $p(k) \ge \frac{2+h(k)}{k-2}$; and
		\item $(k-1)p(k) + (k-3)h(k) \ge k+1$; and
		\item $2(h(k) + 1) + f(k) \le 0$; and
		\item $p(k) + (k-5)h(k) \le k+1$.
	\end{enumerate}
\end{thm}

Before we prove Theorem~\ref{UberTheorem}, we show that two previous results on
this problem follow immediately from this theorem.  Note that conditions
(1)--(5) are the hypotheses of Lemma~\ref{BoundFamilyWithKKMinusOne}.  As a
first test, let $p(k) = 1 - \frac{2}{k-1}$, $f(k) = -2$ and $h(k) = 0$.  Now
the hypotheses of Theorem \ref{UberTheorem} are satisfied when $k\ge7$, and we
get Gallai's bound: $d(G) \ge k-1 + \frac{k-3}{k^2-3}$.  Next, let's use the
Kostochka--Stiebitz bound, that is, $p(k) = \frac{4(k-1)}{k^2 - 3k + 4}$, $f(k)
= -\frac{4(k^2-3k+2)}{k^2-3k+4}$ and $h(k) = \frac{k^2 - 3k}{k^2-3k+4}$. 
Again, the hypotheses of Theorem \ref{UberTheorem} are satisfied when $k \ge 7$
and we get
\[d(G) \ge k-1 + \frac{2(k-2)(k-3)}{(k-1)(k^2 + 3k - 12)}.\]
This is exactly the bound in the paper of Kierstead and the second
author~\cite{OreVizing}.  

Finally, to get our sharpest bound on $d(G)$, we use the bound in Corollary
\ref{SmallP}, that is, $p(k) = \frac{3k-5}{k^2 - 4k + 5}$, $f(k) =
-\frac{2(k-1)(2k-5)}{k^2 - 4k + 5}$, and $h(k) = \frac{k(k-3)}{k^2 - 4k + 5}$. 
The hypotheses of Theorem \ref{UberTheorem} are satisfied when $k\ge7$ and we
get $d(G) \ge k-1 + \frac{(k-3)(2k-5)}{k^3 + k^2 - 15k + 15}.$ This is better
than the bound in~\cite{OreVizing} for $k \ge 7$.  We record this as our main
corollary.

\begin{cor}\label{MainCor}
If $G$ is a $k$-AT-critical graph, with $k\ge 7$, and $G\ne K_k$, then
 \[d(G) \ge k-1 + \frac{(k-3)(2k-5)}{k^3 + k^2 - 15k + 15}.\]
\end{cor}

Now we prove Theorem~\ref{UberTheorem}
\begin{proof}[Proof of Theorem~\ref{UberTheorem}]
Our discharging procedure in the previous section gives charge $\epsilon$ to a component $T$ for every incident edge not ending in a $K_{k-1}$.  The number of such edges is exactly
\[-q(T) + \sum_{v \in V(T)} (k-1 - d_T(v)) = (k-1)\card{T} - 2\size{T} - q(T),\]
so we let $A(T)$ denote this quantity.  When $K_{k-1} \subseteq T$, since
(1)--(5) hold, Lemma~\ref{BoundFamilyWithKKMinusOne} gives
\[2\size{T} \le (k-3 + p(k))\card{T} + f(k) + h(k)q(T).\]
So, when $K_{k-1} \subseteq T$ we get
\begin{align*}
    A(T) & \ge (k-1)\card{T} - q(T) - ((k-3 + p(k))\card{T} + f(k) + h(k)q(T))\\
         &  =(2-p(k))\card{T} - f(k) - (h(k) + 1)q(T).
\end{align*}
Hence, in total $T$ receives charge at least
\begin{align*}
    \epsilon A(T) + \gamma(q(T) - 2) &\ge \epsilon(2-p(k))|T| - \epsilon f(k) - \epsilon (h(k)+1)q(T) +\gamma q(T)-2\gamma \\
    & = \epsilon(2-p(k))|T| + q(T)(\gamma - \epsilon (h(k)+1)) - (2\gamma + \epsilon f(k))
\end{align*}
Our goal is to make $\epsilon(2-p(k))$ as large as possible, while ensuring that the final two terms are nonnegative.  To make the second term 0, we let $\gamma = \epsilon(h(k) + 1)$.  Now the final term becomes $-\epsilon(2(h(k)+1)+f(k))$.
%
%
For simplicity, we have added, as (6), that $2(h(k) + 1) + f(k) \le 0$.  (Since we typically take $h(k) > 0$, as in Corollary~\ref{MainCor}, it is precisely this requirement that necessitates the use of $f(k)$ in Lemma~\ref{BoundFamilyWithKKMinusOne}.)  Thus, $T$ receives charge at least
\[\epsilon\parens{2-p(k)}\card{T},\]
so each of its vertices gets at least $\epsilon(2-p(k))$.
We also need each $k$-vertex to end with enough charge, and each of these loses at most $3\gamma+(k-3)\epsilon$.  So we take
\[1 - (3\gamma + (k-3)\epsilon) = \epsilon\parens{2-p(k)},\]
which gives
\[\epsilon = \frac{1}{k+2 + 3h(k) - p(k)},\]
\[\gamma = \frac{h(k)+1}{k+2 + 3h(k) - p(k)}.\]
Thus, after discharging, each $k$-vertex finishes with charge at least $k-1+\epsilon(2-p(k))$.  The same bound holds for each $(k-1)$-vertex in a component $T$ with a $K_{k-1}$.

When $K_{k-1} \not \subseteq T$, we have $q(T) = 0$.  Applying Lemma
\ref{BoundFamilyWithoutKKMinusOne} with $f(k) = 0$ and $p(k)$ as in the present
theorem, we get 
\[2\size{T} \le (k-3 + p(k))\card{T},\]
and hence
\[A(T) \ge (2-p(k))\card{T}.\]
So $T$ receives sufficient charge.

It remains to check that the $(k+1)^+$-vertices don't give away too much
charge.  Let $v$ be a $(k+1)^+$-vertex. Now $v$ ends with charge at least
\[d(v) - \gamma d(v) = (1-\gamma)d(v) \ge (1-\gamma)(k+1) = (k+1)\frac{k+1 +
2h(k) - p(k)}{k+2 + 3h(k) - p(k)},\]
so we need to satisfy the inequality
\[(k+1)\frac{k+1 + 2h(k) - p(k)}{k+2 + 3h(k) - p(k)} \ge k-1 +
\frac{2-p(k)}{k+2 + 3h(k) - p(k)}.\]
This inequality reduces to
\[p(k) + (k-5)h(k) \le k+1.\]
For simplicity, we have added this as (7), since it is easily satisfied by the
$p$, $f$, and $h$ we want to use.
\end{proof}

The reason that we require $k\ge 7$ in Theorem~\ref{UberTheorem} (and
Corollary~\ref{MainCor}) is that the proof uses
Lemma~\ref{MultipleHighConfigurationEuler}.  However, for $k\in\{5,6\}$,
Lemma~\ref{MultipleHighConfigurationEulerLopsided} can play an analogous role. 
For $k\ge 7$, Lemma~\ref{MultipleHighConfigurationEuler} implies that if $G$
has no reducible configuration, then $B_k(G)$ is 2-degenerate.  For
$k\in\{5,6\}$, Lemma~\ref{MultipleHighConfigurationEulerLopsided} implies that
we can reduce $\B_k(G)$ to the empty graph by repeatedly deleting either a tree
component vertex $v$ with $d_{\B_k(G)}(v)\le 1$ or else a vertex $w$ in
$V(\B_k(G))\cap V(\HH(G))$ with $d_{\B_k(G)}(v)\le 3$.  Thus, in the
discharging, the tree corresponding to $v$ receives charge at least
$\gamma(q(T)-1)$ on edges ending at vertices in $W^k(T)$.  Similarly, each
$k$-vertex gives away charge at most $4\gamma+(k-4)\epsilon$.  Now, to find the
optimal value of $\epsilon$, as in the proof of Theorem~\ref{UberTheorem}, we
solve $(1-(4\gamma+\epsilon(k-4))=(2-p(k))\epsilon$.  This gives $\epsilon =
\frac1{k+2+4h(k)-p(k)}$ and, again, $\gamma=\epsilon(h(k)+1)$.  In place of
hypothesis (6), we have the slightly weaker requirement $h(k)+1+f(k)\le 0$. 
The result is the following theorem and corollary, for $k\in\{5,6\}$.

\begin{thm}\label{Uber56}
	Let $k \in\{5,6\}$ and $\func{p}{\IN}{\IR}$, $\func{f}{\IN}{\IR}$, $\func{h}{\IN}{\IR}$.  If $G$ is a $k$-AT-critical graph, and $G\ne K_k$, then 
	\[d(G) \ge k-1 + \frac{2-p(k)}{k+2 + 4h(k) - p(k)},\]
	whenever $p$, $f$, and $h$ satisfy all of the following conditions:
	\begin{enumerate}
	\item $f(k) \ge (k-1)(1- p(k) - h(k))$; and	
	    \item $p(k) \ge \frac{3}{k-2}$; and
		\item $p(k) \ge h(k) + 5 - k$; and
		\item $p(k) \ge \frac{2+h(k)}{k-2}$; and
		\item $(k-1)p(k) + (k-3)h(k) \ge k+1$; and
		\item $h(k) + 1 + f(k) \le 0$; and
		\item $p(k) + (k-5)h(k) \le k+1$.
	\end{enumerate}
\end{thm}

To get the best bound on $d(G)$, as in Theorem~\ref{UberTheorem}, we use $p(k)
= \frac{3k-5}{k^2 - 4k + 5}$, $f(k) = -\frac{2(k-1)(2k-5)}{k^2 - 4k + 5}$, and
$h(k) = \frac{k(k-3)}{k^2 - 4k + 5}$.
\begin{cor}\label{MinorCor}
If $G$ is a $k$-AT-critical graph, with $k\in\{5,6\}$, and $G\ne K_k$, then
 \[d(G) \ge k-1 + \frac{(k-3)(2k-5)}{k^3 + 2k^2 - 18k + 15}.\]
\end{cor}

\section{Reducible Configurations}
In this section, we collect the three main reducibility lemmas used in the
proofs of Theorems~\ref{UberTheorem} and~\ref{Uber56}.  They were proved
in~\cite{OreVizing}.  Each lemma describes a class of reducible configurations,
and so restricts the structure of AT-critical-graphs.
The first says that no $k$-vertex has three or more neighbors in the same component
$T$ of $\L(G)$.  Further, for each $k$-vertex $v$,  at most one component $T$ of
$\L(G)$ has two neighbors of $v$.

\begin{lem}\label{ConfigurationTypeOneEuler}
Let $k \ge 5$ and let $G$ be a graph with $x \in V(G)$.  
Now $G$ is $f$-AT,
where $f(x) = d_G(x) - 1$ and $f(v) = d_G(v)$ for all $v \in V(G - x)$,
whenever all of the following hold:
\begin{enumerate}
\item $K_k \not \subseteq G$; and
\item $G-x$ has $t$ components $H_1, H_2, \ldots, H_t$, and all are in $\T_k$; and
\item $d_G(v) \leq k - 1$ for all $v \in V(G-x)$; and
\item $\card{N(x) \cap W^k(H_i)} \ge 1$ for $i \in \irange{t}$; and
\item $d_G(x) \ge t+2$.
\end{enumerate}

\end{lem}

To describe reducible configurations with more than one $k$-vertex we need the
following auxiliary bipartite graph, which is a generalization of what we
defined in Section~\ref{discharging-overview}.  For a graph $G$, $\set{X, Y}$ a
partition of $V(G)$ and $k
\ge 4$, let $\B_k(X, Y)$ be the bipartite graph with one part $Y$ and the other
part the components of $G[X]$.  Put an edge between $y \in Y$ and a component
$T$ of $G[X]$ if and only if $N(y) \cap W^k(T) \ne \emptyset$.   The next lemma
tells us that we have a reducible configuration if this bipartite graph has
minimum degree at least three.  In other words, if we have no reducible
configuration, then $\B_k(X,Y)$ is 2-degenerate.

\begin{lem}
\label{MultipleHighConfigurationEuler} 
Let $k\ge7$ and let $G$ be a graph with $Y\subseteq V(G)$.  Now $G$ has an
induced subgraph $G'$ that is $f$-AT, where $f(y)=d_{G'}(y)-1$ for $y\in Y$ and
$f(v)=d_{G'}(v)$ for all $v\in V(G'-Y)$, whenever all of the following hold:
	\begin{enumerate}
		\item $K_{k}\not\subseteq G$; and 
		\item the components of $G-Y$ are in $\T_{k}$; and 
		\item $d_{G}(v)\leq k-1$ for all $v\in V(G-Y)$; and 
		\item with $\B\DefinedAs\B_{k}(V(G-Y),Y)$ we have $\delta(\B)\ge3$. 
	\end{enumerate}
\end{lem}

We also have the following version with asymmetric degree condition on $\B$. 
The point here is that this works for $k \ge 5$.  As we saw in
Theorem~\ref{Uber56} and Corollary~\ref{MinorCor}, the consequence is that we
trade a bit in our size bound for the proof to go through with $k \in
\set{5,6}$.

\begin{lem}
\label{MultipleHighConfigurationEulerLopsided} 
Let $k \ge 5$ and let $G$ be a graph with $Y\subseteq V(G)$.  Now $G$ has an induced
subgraph $G'$ that is $f$-AT where $f(y)=d_{G'}(y)-1$ for $y\in Y$ and
$f(v)=d_{G'}(v)$ for all $v\in V(G'-Y)$ whenever all of the following hold:
	\begin{enumerate}
		\item $K_{k}\not\subseteq G$; and 
		\item the components of $G-Y$ are in $\T_{k}$; and 
		\item $d_{G}(v)\leq k-1$ for all $v\in V(G-Y)$; and 
		\item with $\B \DefinedAs \B_k(V(G-Y), Y)$ we have $d_{\B}(y) \ge 4$ for all $y \in Y$ and $d_{\B}(T) \ge 2$ for all components $T$ of $G-Y$.
	\end{enumerate}
\end{lem}

\bibliographystyle{amsplain}
\bibliography{GraphColoring1}
\end{document}

%% file: m2.tex
\begin{figure}
\centering
\begin{tikzpicture}[scale = 10]
\tikzstyle{VertexStyle} = []
\tikzstyle{EdgeStyle} = []
\tikzstyle{labeledStyle}=[shape = circle, minimum size = 6pt, inner sep = 1.2pt, draw]
\tikzstyle{unlabeledStyle}=[shape = circle, minimum size = 6pt, inner sep = 1.2pt, draw, fill]
\tikzstyle{Qstyle}=[shape = circle,minimum size = 6pt,inner sep = 1.2pt,draw]
\tikzstyle{NotQKkMinusOneStyle}=[shape = circle,minimum size = 6pt,inner sep = 1.2pt,draw]
\tikzstyle{endblockStyle}=[shape = circle,minimum size = 6pt,inner sep = 1.2pt,draw]
\Vertex[style = Qstyle, x = 0.450, y = 0.600, L = \tiny {}]{v0}
\Vertex[style = NotQKkMinusOneStyle, x = 0.550, y = 0.600, L = \tiny {}]{v1}
\Vertex[style = NotQKkMinusOneStyle, x = 0.450, y = 0.500, L = \tiny {}]{v2}
\Vertex[style = NotQKkMinusOneStyle, x = 0.550, y = 0.500, L = \tiny {}]{v3}
\Vertex[style = endblockStyle, x = 0.350, y = 0.450, L = \tiny {}]{v4}
\Vertex[style = endblockStyle, x = 0.300, y = 0.350, L = \tiny {}]{v5}
\Vertex[style = endblockStyle, x = 0.250, y = 0.450, L = \tiny {}]{v6}
\Vertex[style = endblockStyle, x = 0.600, y = 0.300, L = \tiny {}]{v7}
\Vertex[style = endblockStyle, x = 0.500, y = 0.300, L = \tiny {}]{v8}
\Vertex[style = endblockStyle, x = 0.550, y = 0.400, L = \tiny {}]{v9}
\Vertex[style = NotQKkMinusOneStyle, x = 0.750, y = 0.500, L = \tiny {}]{v10}
\Vertex[style = NotQKkMinusOneStyle, x = 0.850, y = 0.500, L = \tiny {}]{v11}
\Vertex[style = NotQKkMinusOneStyle, x = 0.750, y = 0.600, L = \tiny {}]{v12}
\Vertex[style = Qstyle, x = 0.850, y = 0.600, L = \tiny {}]{v13}
\Vertex[style = endblockStyle, x = 1.050, y = 0.450, L = \tiny {}]{v14}
\Vertex[style = endblockStyle, x = 1.000, y = 0.350, L = \tiny {}]{v15}
\Vertex[style = endblockStyle, x = 0.950, y = 0.450, L = \tiny {}]{v16}
\Vertex[style = endblockStyle, x = 0.800, y = 0.300, L = \tiny {}]{v17}
\Vertex[style = endblockStyle, x = 0.700, y = 0.300, L = \tiny {}]{v18}
\Vertex[style = endblockStyle, x = 0.750, y = 0.400, L = \tiny {}]{v19}
\Edge[label = \tiny {}, labelstyle={auto=right, fill=none}](v1)(v0)
\Edge[label = \tiny {}, labelstyle={auto=right, fill=none}](v1)(v2)
\Edge[label = \tiny {}, labelstyle={auto=right, fill=none}](v1)(v3)
\Edge[label = \tiny {}, labelstyle={auto=right, fill=none}](v2)(v0)
\Edge[label = \tiny {}, labelstyle={auto=right, fill=none}](v2)(v4)
\Edge[label = \tiny {}, labelstyle={auto=right, fill=none}](v3)(v0)
\Edge[label = \tiny {}, labelstyle={auto=right, fill=none}](v3)(v2)
\Edge[label = \tiny {}, labelstyle={auto=right, fill=none}](v4)(v5)
\Edge[label = \tiny {}, labelstyle={auto=right, fill=none}](v4)(v6)
\Edge[label = \tiny {}, labelstyle={auto=right, fill=none}](v5)(v6)
\Edge[label = \tiny {}, labelstyle={auto=right, fill=none}](v7)(v8)
\Edge[label = \tiny {}, labelstyle={auto=right, fill=none}](v7)(v9)
\Edge[label = \tiny {}, labelstyle={auto=right, fill=none}](v8)(v9)
\Edge[label = \tiny {}, labelstyle={auto=right, fill=none}](v9)(v3)
\Edge[label = \tiny {}, labelstyle={auto=right, fill=none}](v11)(v10)
\Edge[label = \tiny {}, labelstyle={auto=right, fill=none}](v11)(v12)
\Edge[label = \tiny {}, labelstyle={auto=right, fill=none}](v11)(v13)
\Edge[label = \tiny {}, labelstyle={auto=right, fill=none}](v12)(v1)
\Edge[label = \tiny {}, labelstyle={auto=right, fill=none}](v12)(v10)
\Edge[label = \tiny {}, labelstyle={auto=right, fill=none}](v13)(v10)
\Edge[label = \tiny {}, labelstyle={auto=right, fill=none}](v13)(v12)
\Edge[label = \tiny {}, labelstyle={auto=right, fill=none}](v14)(v15)
\Edge[label = \tiny {}, labelstyle={auto=right, fill=none}](v14)(v16)
\Edge[label = \tiny {}, labelstyle={auto=right, fill=none}](v15)(v16)
\Edge[label = \tiny {}, labelstyle={auto=right, fill=none}](v16)(v11)
\Edge[label = \tiny {}, labelstyle={auto=right, fill=none}](v17)(v18)
\Edge[label = \tiny {}, labelstyle={auto=right, fill=none}](v17)(v19)
\Edge[label = \tiny {}, labelstyle={auto=right, fill=none}](v18)(v19)
\Edge[label = \tiny {}, labelstyle={auto=right, fill=none}](v19)(v10)
\end{tikzpicture}
\caption{The construction when $k=5$ and $m=2$.}
\end{figure}

%% file: m3.tex
\begin{figure}
\centering
\begin{tikzpicture}[scale = 10]
\tikzstyle{VertexStyle} = []
\tikzstyle{EdgeStyle} = []
\tikzstyle{labeledStyle}=[shape = circle, minimum size = 6pt, inner sep = 1.2pt, draw]
\tikzstyle{unlabeledStyle}=[shape = circle, minimum size = 6pt, inner sep = 1.2pt, draw, fill]
\tikzstyle{NotQKkMinusOneStyle}=[shape = circle,minimum size = 6pt,inner sep = 1.2pt,draw]
\tikzstyle{Qstyle}=[shape = circle,minimum size = 6pt,inner sep = 1.2pt,draw]
\tikzstyle{endblockStyle}=[shape = circle,minimum size = 6pt,inner sep = 1.2pt,draw]
\Vertex[style = NotQKkMinusOneStyle, x = 0.650, y = 0.550, L = \tiny {}]{v0}
\Vertex[style = NotQKkMinusOneStyle, x = 0.750, y = 0.550, L = \tiny {}]{v1}
\Vertex[style = NotQKkMinusOneStyle, x = 0.650, y = 0.650, L = \tiny {}]{v2}
\Vertex[style = NotQKkMinusOneStyle, x = 0.750, y = 0.650, L = \tiny {}]{v3}
\Vertex[style = Qstyle, x = 0.450, y = 0.550, L = \tiny {}]{v4}
\Vertex[style = NotQKkMinusOneStyle, x = 0.550, y = 0.550, L = \tiny {}]{v5}
\Vertex[style = NotQKkMinusOneStyle, x = 0.450, y = 0.450, L = \tiny {}]{v6}
\Vertex[style = NotQKkMinusOneStyle, x = 0.550, y = 0.450, L = \tiny {}]{v7}
\Vertex[style = endblockStyle, x = 0.350, y = 0.400, L = \tiny {}]{v8}
\Vertex[style = endblockStyle, x = 0.300, y = 0.300, L = \tiny {}]{v9}
\Vertex[style = endblockStyle, x = 0.250, y = 0.400, L = \tiny {}]{v10}
\Vertex[style = endblockStyle, x = 0.600, y = 0.250, L = \tiny {}]{v11}
\Vertex[style = endblockStyle, x = 0.500, y = 0.250, L = \tiny {}]{v12}
\Vertex[style = endblockStyle, x = 0.550, y = 0.350, L = \tiny {}]{v13}
\Vertex[style = endblockStyle, x = 0.500, y = 0.800, L = \tiny {}]{v14}
\Vertex[style = endblockStyle, x = 0.450, y = 0.700, L = \tiny {}]{v15}
\Vertex[style = endblockStyle, x = 0.550, y = 0.700, L = \tiny {}]{v16}
\Vertex[style = NotQKkMinusOneStyle, x = 0.850, y = 0.450, L = \tiny {}]{v17}
\Vertex[style = NotQKkMinusOneStyle, x = 0.950, y = 0.450, L = \tiny {}]{v18}
\Vertex[style = NotQKkMinusOneStyle, x = 0.850, y = 0.550, L = \tiny {}]{v19}
\Vertex[style = Qstyle, x = 0.950, y = 0.550, L = \tiny {}]{v20}
\Vertex[style = endblockStyle, x = 1.150, y = 0.400, L = \tiny {}]{v21}
\Vertex[style = endblockStyle, x = 1.100, y = 0.300, L = \tiny {}]{v22}
\Vertex[style = endblockStyle, x = 1.050, y = 0.400, L = \tiny {}]{v23}
\Vertex[style = endblockStyle, x = 0.900, y = 0.250, L = \tiny {}]{v24}
\Vertex[style = endblockStyle, x = 0.800, y = 0.250, L = \tiny {}]{v25}
\Vertex[style = endblockStyle, x = 0.850, y = 0.350, L = \tiny {}]{v26}
\Vertex[style = endblockStyle, x = 0.900, y = 0.800, L = \tiny {}]{v27}
\Vertex[style = endblockStyle, x = 0.850, y = 0.700, L = \tiny {}]{v28}
\Vertex[style = endblockStyle, x = 0.950, y = 0.700, L = \tiny {}]{v29}
\Edge[label = \tiny {}, labelstyle={auto=right, fill=none}](v0)(v5)
\Edge[label = \tiny {}, labelstyle={auto=right, fill=none}](v1)(v0)
\Edge[label = \tiny {}, labelstyle={auto=right, fill=none}](v1)(v2)
\Edge[label = \tiny {}, labelstyle={auto=right, fill=none}](v1)(v3)
\Edge[label = \tiny {}, labelstyle={auto=right, fill=none}](v1)(v19)
\Edge[label = \tiny {}, labelstyle={auto=right, fill=none}](v2)(v0)
\Edge[label = \tiny {}, labelstyle={auto=right, fill=none}](v3)(v0)
\Edge[label = \tiny {}, labelstyle={auto=right, fill=none}](v3)(v2)
\Edge[label = \tiny {}, labelstyle={auto=right, fill=none}](v5)(v4)
\Edge[label = \tiny {}, labelstyle={auto=right, fill=none}](v5)(v6)
\Edge[label = \tiny {}, labelstyle={auto=right, fill=none}](v5)(v7)
\Edge[label = \tiny {}, labelstyle={auto=right, fill=none}](v6)(v4)
\Edge[label = \tiny {}, labelstyle={auto=right, fill=none}](v6)(v8)
\Edge[label = \tiny {}, labelstyle={auto=right, fill=none}](v7)(v4)
\Edge[label = \tiny {}, labelstyle={auto=right, fill=none}](v7)(v6)
\Edge[label = \tiny {}, labelstyle={auto=right, fill=none}](v8)(v9)
\Edge[label = \tiny {}, labelstyle={auto=right, fill=none}](v8)(v10)
\Edge[label = \tiny {}, labelstyle={auto=right, fill=none}](v9)(v10)
\Edge[label = \tiny {}, labelstyle={auto=right, fill=none}](v11)(v12)
\Edge[label = \tiny {}, labelstyle={auto=right, fill=none}](v11)(v13)
\Edge[label = \tiny {}, labelstyle={auto=right, fill=none}](v12)(v13)
\Edge[label = \tiny {}, labelstyle={auto=right, fill=none}](v13)(v7)
\Edge[label = \tiny {}, labelstyle={auto=right, fill=none}](v14)(v15)
\Edge[label = \tiny {}, labelstyle={auto=right, fill=none}](v14)(v16)
\Edge[label = \tiny {}, labelstyle={auto=right, fill=none}](v15)(v16)
\Edge[label = \tiny {}, labelstyle={auto=right, fill=none}](v16)(v2)
\Edge[label = \tiny {}, labelstyle={auto=right, fill=none}](v18)(v17)
\Edge[label = \tiny {}, labelstyle={auto=right, fill=none}](v18)(v19)
\Edge[label = \tiny {}, labelstyle={auto=right, fill=none}](v18)(v20)
\Edge[label = \tiny {}, labelstyle={auto=right, fill=none}](v19)(v17)
\Edge[label = \tiny {}, labelstyle={auto=right, fill=none}](v20)(v17)
\Edge[label = \tiny {}, labelstyle={auto=right, fill=none}](v20)(v19)
\Edge[label = \tiny {}, labelstyle={auto=right, fill=none}](v21)(v22)
\Edge[label = \tiny {}, labelstyle={auto=right, fill=none}](v21)(v23)
\Edge[label = \tiny {}, labelstyle={auto=right, fill=none}](v22)(v23)
\Edge[label = \tiny {}, labelstyle={auto=right, fill=none}](v24)(v25)
\Edge[label = \tiny {}, labelstyle={auto=right, fill=none}](v24)(v26)
\Edge[label = \tiny {}, labelstyle={auto=right, fill=none}](v25)(v26)
\Edge[label = \tiny {}, labelstyle={auto=right, fill=none}](v26)(v17)
\Edge[label = \tiny {}, labelstyle={auto=right, fill=none}](v23)(v18)
\Edge[label = \tiny {}, labelstyle={auto=right, fill=none}](v27)(v28)
\Edge[label = \tiny {}, labelstyle={auto=right, fill=none}](v27)(v29)
\Edge[label = \tiny {}, labelstyle={auto=right, fill=none}](v28)(v29)
\Edge[label = \tiny {}, labelstyle={auto=right, fill=none}](v28)(v3)
\end{tikzpicture}
\caption{The construction when $k=5$ and $m=3$.}
\end{figure}

%% file: EdgeBoundDischargingFinal.bbl
\providecommand{\bysame}{\leavevmode\hbox to3em{\hrulefill}\thinspace}
\providecommand{\MR}{\relax\ifhmode\unskip\space\fi MR }
\providecommand{\MRhref}[2]{%
  \href{http://www.ams.org/mathscinet-getitem?mr=#1}{#2}
}
\providecommand{\href}[2]{#2}
\begin{thebibliography}{10}

\bibitem{Alon1992125}
N.~Alon and M.~Tarsi, \emph{Colorings and orientations of graphs},
  Combinatorica \textbf{12} (1992), no.~2, 125--134.

\bibitem{erdos1979choosability}
P.~Erd\H{o}s, A.L. Rubin, and H.~Taylor, \emph{{Choosability in graphs}}, Proc.
  West Coast Conf. on Combinatorics, Graph Theory and Computing, Congressus
  Numerantium, vol.~26, 1979, pp.~125--157.

\bibitem{gallai1963kritische}
T.~Gallai, \emph{{Kritische Graphen I.}}, Publ. Math. Inst. Hungar. Acad. Sci
  \textbf{8} (1963), 165--192 (in German).

\bibitem{Hladky}
J.~Hladk{\'y}, D.~Kr{\'a}l, and U.~Schauz, \emph{Brooks' theorem via the
  {A}lon-{T}arsi theorem}, Discrete Math. \textbf{310} (2010), no.~23,
  3426--3428. \MR{2721105 (2012a:05115)}

\bibitem{kierstead2009ore}
H.A. Kierstead and A.V. Kostochka, \emph{{Ore-type versions of Brooks'
  theorem}}, Journal of Combinatorial Theory, Series B \textbf{99} (2009),
  no.~2, 298--305.

\bibitem{OreVizing}
H.A. Kierstead and L.~Rabern, \emph{Improved lower bounds on the number of
  edges in list critical and online list critical graphs}, arXiv preprint
  arXiv:1406.7355 (2014).

\bibitem{krs_one}
A.V. Kostochka, L.~Rabern, and M.~Stiebitz, \emph{{Graphs with chromatic number
  close to maximum degree}}, Discrete Mathematics \textbf{312} (2012), no.~6,
  1273--1281.

\bibitem{kostochkastiebitzedgesincriticalgraph}
A.V. Kostochka and M.~Stiebitz, \emph{A new lower bound on the number of edges
  in colour-critical graphs and hypergraphs}, Journal of Combinatorial Theory,
  Series B \textbf{87} (2003), no.~2, 374--402.

\bibitem{kostochka2012oregrotsch}
A.V. Kostochka and M.~Yancey, \emph{Ore's conjecture for {$k=4$} and
  {G}r\"otzsch's theorem}, Combinatorica \textbf{34} (2014), no.~3, 323--329.
  \MR{3223967}

\bibitem{kostochkayancey2012ore}
\bysame, \emph{Ore's conjecture on color-critical graphs is almost true}, J.
  Combin. Theory Ser. B \textbf{109} (2014), 73--101. \MR{3269903}

\bibitem{krivelevich1997minimal}
M.~Krivelevich, \emph{On the minimal number of edges in color-critical graphs},
  Combinatorica \textbf{17} (1997), no.~3, 401--426.

\bibitem{rabern2010a}
L.~Rabern, \emph{{$\Delta$-critical graphs with small high vertex cliques}},
  Journal of Combinatorial Theory, Series B \textbf{102} (2012), no.~1,
  126--130.

\bibitem{riasat2012critically}
A.~Riasat and U.~Schauz, \emph{Critically paintable, choosable or colorable
  graphs}, Discrete Mathematics \textbf{312} (2012), no.~22, 3373--3383.

\bibitem{schauz2009mr}
U.~Schauz, \emph{{Mr. Paint and Mrs. Correct}}, {The Electronic Journal of
  Combinatorics} \textbf{16} (2009), no.~1, R77.

\bibitem{schauz2010flexible}
\bysame, \emph{{Flexible color lists in Alon and Tarsi's theorem, and time
  scheduling with unreliable participants}}, {The Electronic Journal of
  Combinatorics} \textbf{17} (2010), no.~1, R13.

\bibitem{stiebitz1982proof}
M.~Stiebitz, \emph{{Proof of a conjecture of T. Gallai concerning connectivity
  properties of colour-critical graphs}}, Combinatorica \textbf{2} (1982),
  no.~3, 315--323.

\bibitem{vizing1976}
V.G. Vizing, \emph{{Vextex coloring with given colors}}, Metody Diskretn. Anal.
  \textbf{29} (1976), 3--10 (in Russian).

\bibitem{zhu2009online}
X.~Zhu, \emph{On-line list colouring of graphs}, {The Electronic Journal of
  Combinatorics} \textbf{16} (2009), no.~1, R127.

\end{thebibliography}
